\documentclass[12pt]{article}

\usepackage{epsfig}
\usepackage{amsfonts}
\usepackage{amssymb}
\usepackage{amstext}
\usepackage{amsmath}
\usepackage{xspace}
\usepackage{theorem}
\usepackage{xcolor}
\usepackage{color}
\usepackage{graphicx}
\usepackage{tikz}
\usetikzlibrary{calc}
\usepackage{ifthen}
\usepackage{hyperref}

\newcommand{\junk}[1]{}

\makeatletter
\newenvironment{proof}{{\bf Proof:  }}{\hfill\rule{2mm}{2mm}}
\newenvironment{proofof}[1]{{\bf Proof of #1:  }}{\hfill\rule{2mm}{2mm}}

\newtheorem{theorem}{Theorem}
\newtheorem{lemma}[theorem]{Lemma}
\newtheorem{proposition}[theorem]{Proposition}

\newtheorem{corollary}[theorem]{Corollary}

\newtheorem{problem}[theorem]{Problem}
\newtheorem{fact}[theorem]{Fact}

\newcommand{\Z}{\ensuremath{\mathbb Z}}

\title{Polychromatic Colorings on the Hypercube}

\author{
John Goldwasser\thanks{West Virginia University, \texttt{jgoldwas@math.wvu.edu}},
Bernard Lidicky\thanks{Iowa State University, \texttt{lidicky@iastate.edu}. Research of this author is supported in part by NSF grants DMS-1266016 and DMS-1600390.},
Ryan R. Martin\thanks{Iowa State University, \texttt{rymartin@iastate.edu}. 
Research of this author is supported in part by National Security Agency grant H98230-13-1-0226 and by Simons Foundation grant \#353292.},\\ 
David Offner\thanks{Westminster College, \texttt{offnerde@westminster.edu}}, 
John Talbot\thanks{University College London, \texttt{j.talbot@ucl.ac.uk}},
 and Michael Young\thanks{Iowa State University, \texttt{myoung@iastate.edu}}
}

\begin{document}

\maketitle

\begin{abstract}
Given a subgraph $G$ of the hypercube $Q_n$, a coloring of the edges of $Q_n$ such that every embedding of $G$ contains an edge of every color is called a $G$-polychromatic coloring. The maximum number of colors with which it is possible to $G$-polychromatically color the edges of \emph{any} hypercube is called the polychromatic number of $G$. To determine polychromatic numbers, it is only necessary to consider a specific type of coloring, which we call simple.  The main tool for finding upper bounds on polychromatic numbers is to translate the question of polychromatically coloring the hypercube so every embedding of a graph $G$ contains every color into a question of coloring the 2-dimensional grid  so that every so-called \textit{shape sequence} corresponding to $G$  contains every color. After surveying the tools for finding polychromatic numbers, we apply these techniques to find polychromatic numbers of a class of graphs called punctured hypercubes.  We also consider the problem of finding polychromatic numbers in the setting where larger subcubes of the hypercube are colored.  We exhibit two new constructions which show that this problem is not a straightforward generalization of the edge coloring problem.
\end{abstract}


\section{Introduction}
For $n \in \Z$, $n \ge 1$,  the \emph{$n$-dimensional hypercube}, denoted
by $Q_n$, is the graph with $V(Q_n) = \{0,1\}^n$, and edges
between vertices which differ in exactly one coordinate. For any graphs $G$, $H$, a subgraph of $H$ isomorphic to $G$ is called an \emph{embedding} of $G$ in $H$.  Given a set $R$ of $r$ colors, an edge coloring of a graph $G$ with $r$ colors is a surjective function $\chi:E(G) \rightarrow R$ assigning a color to each edge of $G$. All colorings of graphs
will refer to edge colorings, unless otherwise noted. Given a graph $G$, an edge coloring of a hypercube with $r$ colors such that every
embedding of $G$ contains an edge of every color is called a
\emph{$G$-polychromatic} $r$-coloring, and we denote by $p(G)$ the maximum
number of colors with which it is possible to $G$-polychromatically
color the edges of \emph{any} hypercube. Call $p(G)$ the \textit{polychromatic number} of $G$.

Motivated by Tur\'an type problems on the hypercube, Alon, Krech, and Szab\'o \cite{AKS07} introduced the notion of polychromatic coloring on the hypercube and proved bounds for the polychromatic number of $Q_d$.
\begin{theorem}[Alon, Krech, and Szab\'o~\cite{AKS07}]\label{AKS}
For all $d \ge 1$,
\[ \binom{d+1}{2} \ge p(Q_d) \ge \begin{cases} 
  \frac{(d+1)^2}{4} \text{ if $d$ is odd } \\ 
  \frac{d(d+2)}{4} \text{ if $d$ is even.}
  \end{cases} \]
\end{theorem}

The exact value of the polychromatic number of $Q_d$ was determined in \cite{Off08}.
\begin{theorem}[Offner~\cite{Off08}]\label{pd}
For all $d \ge 1$,
  \[ p(Q_d) = \begin{cases}
  \frac{(d+1)^2}{4} \text{ if $d$ is odd } \\ 
  \frac{d(d+2)}{4} \text{ if $d$ is even.}
  \end{cases} \]
\end{theorem}

Prior to the work of Alon, Krech, and Szab\'o \cite{AKS07}, coloring arguments had also been used to give bounds on Tur\'an type problems on the hypercube, for example by Conder~\cite{Con93} and Axenovich and Martin~\cite{AM06}. In \cite{Off09}, a condition was given which, if satisfied by a graph $G$, implies $p(G) \ge 3$.


In this paper we begin by surveying what is known about polychromatic colorings on the hypercube. In Section~\ref{simpleSection} we establish that when studying polychromatic colorings, we need only consider a specific type of coloring called a \textit{simple} coloring (such colorings were called \emph{Ramsey} in \cite{Off08}). In Section~\ref{boupg}, we use the idea of simple coloring to transform the problem of edge coloring the hypercube so that a given subgraph is polychromatic to one of coloring a rectangular grid so that a collection of subsets is polychromatic.  In this context, Lemma~\ref{pig} provides the key insight to prove upper bounds on polychromatic numbers.  All known lower bounds for polychromatic numbers come from explicit constructions, and at the end of the section we give an example by proving the lower bounds in Theorems~\ref{AKS} and \ref{pd}.

Following this survey, we show in Section~\ref{presults} how to use these methods to determine the value of $p(G)$ for some graphs $G$ where $p(G)$ was not previously known, for example hypercubes with one edge or vertex deleted. We call these graphs punctured hypercubes.  Theorem~\ref{qmv} gives the polychromatic number for any odd-dimensional punctured hypercube, and Theorems~\ref{p4mv} and \ref{p4me} give the polychromatic number for punctured $Q_4$'s.  For even-dimensional punctured hypercubes with dimension greater than 4, Theorem~\ref{pq2kmv} provides a lower bound on the polychromatic number, but we can not determine it exactly. The current best bounds are summarized in Corollary~\ref{evensum}. The section concludes with suggestions for future research.


Section~\ref{generalSection} concerns a generalization of the problem proposed by Alon, Krech, and Szab\'o \cite{AKS07} where instead of edges, subcubes of a fixed dimension are colored. Previously, \"Ozkahya and Stanton \cite{OS11} had generalized the bounds given in Theorem~\ref{AKS} to this setting.  If this more general problem were a straightforward generalization of the edge-coloring problem, the polychromatic number would be equal to the lower bound.  However Theorems~\ref{p233} and \ref{p24} provide two constructions that show this is not the case, and thus new ideas will be required to determine polychromatic numbers in this setting.

\junk{
Finally, in Section~\ref{embedSection}, we mention an alternate definition of embedding that we call \textit{rigid embeddings}, used for example in \cite{GT12}, that may provide an interesting context for polychromatic coloring problems.
}

\subsection{Notation for Hypercubes}
We refer to the $n$ coordinates of a vertex as \emph{bits}, and given an edge
$\{x,y\}$, we refer to the unique bit where $x_i \neq y_i$ as the \emph{flip
bit}. We represent an edge of $Q_n$ by an $n$-bit vector with a star in the
flip bit.  For example, in $Q_4$, we represent the edge between vertices
$[0100]$ and $[0101]$ by $[010{*}]$. Similarly, we represent an embedding of $Q_d$ in $Q_n$ by
an $n$-bit vector with stars in $d$ coordinates.  For instance $[1{*}00{*}]$ is
the embedding of $Q_2$ in $Q_5$ with vertices $\{[10000], [11000], [10001], [11001]\}$ and
edges $\{[1{*}000], [1000{*}], [1{*}001], [1100{*}]\}$. We call edges with the same
flip bit \emph{parallel}, and the class of edges with flip bit $i$ the
\emph{$i^{th}$ parallel class}.  For an edge $e \in E(Q_n)$ with
flip bit $j$ define the prefix sum $l(e)= \sum_{i=1}^{j-1} x_i$ and postfix sum $r(e)=\sum_{i=j+1}^n x_i$.  


\section{Simple Colorings}\label{simpleSection}
Recall that for an edge $e \in E(Q_n)$, $l(e)$ is the number of 1's to the left
of the star in $e$, and $r(e)$ is the number of 1's to the right. Call a
coloring $\chi$ of the hypercube \emph{simple} if $\chi(e)$ is determined by
$l(e)$ and $r(e)$ (such colorings were called \emph{Ramsey} in
\cite{Off08}).  The following lemma tells us that when
studying polychromatic colorings on the hypercube, we need only consider simple ones. The proof is essentially from \cite{Off08}, building on ideas from \cite{AKS07}.

\begin{lemma}\label{simple}
Let $k \ge 1$ and $G$ be a subgraph of $Q_k$.  If $p(G) = r$, then there is a simple $G$-polychromatic $r$-coloring on $Q_k$.
\end{lemma}

\begin{proof}
Fix $k$. We show that if $n$ is sufficiently large and $Q_n$
has a $G$-polychromatic $r$-coloring, then it contains a subgraph $Q_k$ with a
simple coloring.

Suppose that $n$ is large and $\chi$ is a $G$-polychromatic $r$-coloring of $Q_n$.  We will use  Ramsey's theorem for $k$-uniform hypergraphs with
$r^{k2^{k-1}}$ colors.  We define a $r^{k2^{k-1}}$-coloring of the $k$-subsets
of $[n]$.  Fix an arbitrary ordering of the edges of $Q_k$.  For an arbitrary
subset $S$ of the indices, define $cube(S)$ to be the subcube whose ${*}$
coordinates are at the positions of $S$ and all other coordinates are 0.  Let
$S$ be a $k$-subset of $[n]$, and define the color of $S$ to be the vector
whose coordinates are the $\chi$-values of the edges of the $k$-dimensional
subcube $cube(S)$ (according to our fixed ordering of the edges of $Q_k$).  By
Ramsey's theorem, if $n$ is large enough, there is a set $T \subseteq [n]$ of
$k^2+k-1$ coordinates such that the color-vector is the same for any
$k$-subset of $T$.  Fix a set $S$ of $k$ particular coordinates from $T$:
those which are the $(ik)$th elements of $T$ for $ i \in [k]$.  

We show the coloring of $cube(S)$ is simple. Let $e_1$ and $e_2$ be two edges
of $cube(S)$ such that $l(e_1) = l(e_2)$ and $r(e_1) = r(e_2)$.  Since there
are at least $k-1$ elements in $T$ in between each coordinate of $S$, as well as $k-1$ elements to the left of the first coordinate of $S$ and to the right of the last coordinate of $S$, there is a set of $k$ coordinates $S' \subseteq T$ and an edge
$e_3$ of $cube(S')$ such that
\begin{enumerate}
\item[(i)] $e_3$ is the same edge when restricted to $S$ as $e_1$ and 
\item[(ii)] $e_3$ occupies the same position in the ordering of edges in
  $cube(S')$ as $e_2$ occupies in $cube(S)$.
\end{enumerate}
Thus $\chi(e_1) = \chi(e_3) = \chi(e_2)$, so the coloring of $cube(S)$ is a
simple $G$-polychromatic $r$-coloring.  

For example, suppose $k=6$ and ignoring all coordinates not in $T$,
suppose \[\begin{split}
e_1 &= xxxxx0xxxxx1xxxxx1xxxxx{*}xxxxx0xxxxx1xxxxx\\
e_2 &= xxxxx1xxxxx1xxxxx{*}xxxxx0xxxxx1xxxxx0xxxxx\\
\end{split} \] where the coordinates in $T$ but not $S$ are represented by
$x$.  Then a possibility for $e_3$ is
\[e_3=xxxxxxxxxxx1xxxxx1xxxxx{*}xxxxx0xxxxx10xxxx.\] 
\end{proof}

\section{Techniques for Finding Bounds on $p(G)$}\label{boupg}
In a simple coloring of the hypercube, we refer to all edges $e$ with the same value of $(l(e), r(e))$ as a \textit{color class}.  For example, all edges $e$ with $l(e) =2$ and $r(e) = 5$ are in color class $(2,5)$. We begin with an elementary example of how Lemma~\ref{simple} allows us to prove upper bounds on polychromatic numbers.

\begin{proposition}\label{q3-v}
Denote by $Q_3 \setminus v$ the graph $Q_3$ with one vertex deleted. Then
$p(Q_3 \setminus v) \le 3$.
\end{proposition}
\begin{proof}
By Lemma~\ref{simple}, we need to consider only simple colorings.
Consider the embedding of $Q_3 \setminus v$ with the vertex $[1110000\ldots]$
deleted from the cube $[{*}{*}{*}0000\ldots]$.  This graph has edges in only
three color classes, $(0,0)$, $(1,0)$, and $(0,1)$, and thus can only contain three colors in a simple coloring. 
\end{proof}

This example illustrates a general scheme for proving upper bounds on
$p(G)$: Given a graph $G$, show that in an arbitrary simple coloring
there is some embedding of $G$ in $Q_n$ that contains edges in only a small number of color classes.  For instance, applying the argument of
Proposition~\ref{q3-v} to $Q_d$ gives the upper bound in
Theorem~\ref{AKS}.  To do better, we need Lemma~\ref{pig}.

Arrange the set of color classes in a rectangular grid, with the $i$th \emph{row} containing the color classes $(a,b)$, with $a+b=i$, and the $i$th \emph{column} containing classes of the form $(i,j)$, as shown in  Figure~\ref{ccs}. We translate the question of polychromatically coloring the hypercube so every embedding of a graph $G$ contains every color into a question of coloring the grid of color classes so that every so-called \textit{shape sequence} corresponding to $G$ (which is defined below) contains every color. 

\begin{figure}
\begin{center}
\small
\begin{tikzpicture}[scale=0.9]
\draw (0,0) grid (7,-6);
  \foreach \i in {0,...,5}{
  \foreach \j in {0,...,5} {
  \pgfmathtruncatemacro{\x}{0-\i+\j} 
  \ifthenelse{\x > -1}{  \draw (\i+0.5,-\j-0.5) node{(\i,\x)} (\i,-\j-1) rectangle (\i+1,-\j);  }{}
}
  \draw (\i+0.5,-6-0.5) node{$\vdots$} (\i,-6-1) rectangle (\i+1,-6); 
}
  \draw (6+0.5,-6-0.5) node{$\ddots$} (6,-6-1) rectangle (6+1,-6); 
  
\end{tikzpicture}
\end{center}
\caption{Initial part of the grid of color classes.}\label{ccs}
\end{figure}

Define a \emph{region} of the grid to be all color classes
contained in some consecutive rows and consecutive columns. A \emph{shape} is a finite set of elements of the grid.  Two
shapes are \emph{congruent} if one is a translation of the other, i.e. if
$S=\{(a_1, b_1),(a_2, b_2),\ldots,(a_k, b_k)\}$ then $S' \cong S$ if and only
if $S' = \{(a_1+i, b_1+j),(a_2+i, b_2+j),\ldots,(a_k+i, b_k+j)\}$ for some
$i,j \in \Z$. The width $w(S)$ of a shape $S = \{(a_1, b_1),(a_2, b_2),\ldots,(a_k,
b_k)\}$ is given by $\max_{i,j} |a_i-a_j|$. We say $S$ is located at the
column of its leftmost element, i.e. $S$ is located at column $\min_i
(a_i)$. A \emph{shape list} is a finite list of shapes $S_1,\ldots, S_k$, with
the restriction that if $i<j$ then $S_i$ is not to the right of $S_j$. Two
shape lists are \emph{congruent} if each contains the same number of shapes,
and corresponding shapes in the lists are congruent and are horizontal
translations of each other.  A \emph{shape sequence} $\mathcal{S}$ is the set
of all shape lists congruent to a specific list. An \emph{instance} of a shape
sequence $\mathcal{S}$ is one particular list--when the context is clear, we
will not always distinguish between a shape sequence and an instance of a
shape sequence, since specifying any instance determines all other instances of the
sequence (see Figure~\ref{instances}).  Let the width $w(\mathcal{S})$ of a sequence equal the width of its widest shape.  For the height $h(\mathcal{S})$ of a sequence, if $i_s$ is the
smallest row index where some shape in $\mathcal{S}$ contains some element,
and if $i_l$ is the largest row index where some shape in $\mathcal{S}$
contains some element, let $h(\mathcal{S}) = i_l - i_s +1$.  Finally, given a
shape sequence $\mathcal{S}$, let $p(\mathcal{S})$ be the maximum number of
colors such that for any rectangular grid, there is a coloring of the
elements of the grid so that every instance of $\mathcal{S}$ contains an
element of every color.

\begin{figure}

\begin{center}
\newcommand{\inS}[3]{\fill[gray!60,thick]  (#1+#3,-#2-1) rectangle (#1+1+#3,-#2);  }
\begin{tikzpicture}[scale=0.6]
\foreach \x/\y in {0/0, 0/1, 0/2}{\inS{\x}{\y}{1}}
\foreach \x/\y in {0/0, 0/1, 1/1, 1/2}{\inS{\x}{\y}{3}}
\foreach \x/\y in {0/0, 1/1, 2/2}{\inS{\x}{\y}{7}}
\draw (0,0) grid (10,-3);
\end{tikzpicture}
\hskip 1em
\begin{tikzpicture}[scale=0.6]
\foreach \x/\y in {0/0, 0/1, 0/2}{\inS{\x}{\y}{2}}
\foreach \x/\y in {0/0, 0/1, 1/1, 1/2}{\inS{\x}{\y}{4}}
\foreach \x/\y in {0/0, 1/1, 2/2}{\inS{\x}{\y}{7}}
\draw (0,0) grid (10,-3);
\end{tikzpicture}
\vskip 1em
\begin{tikzpicture}[scale=0.6]
\foreach \x/\y in {0/0, 0/1, 0/2}{\inS{\x}{\y}{3}}
\foreach \x/\y in {0/0, 0/1, 1/1, 1/2}{\inS{\x}{\y}{3}}
\foreach \x/\y in {0/0, 1/1, 2/2}{\inS{\x}{\y}{6}}
\draw (0,0) grid (10,-3);
\end{tikzpicture}
\hskip 1em
\begin{tikzpicture}[scale=0.6]
\foreach \x/\y in {0/0, 0/1, 0/2}{\inS{\x}{\y}{3}}
\foreach \x/\y in {0/0, 0/1, 1/1, 1/2}{\inS{\x}{\y}{3}}
\foreach \x/\y in {0/0, 1/1, 2/2}{\inS{\x}{\y}{3}}
\draw (0,0) grid (10,-3);
\end{tikzpicture}
\end{center}
\caption{Four instances of a given shape sequence (for $Q_3$). In the two instances at the
  bottom, the shapes overlap, which is allowed as long as they remain in
  order.}\label{instances}
\end{figure}
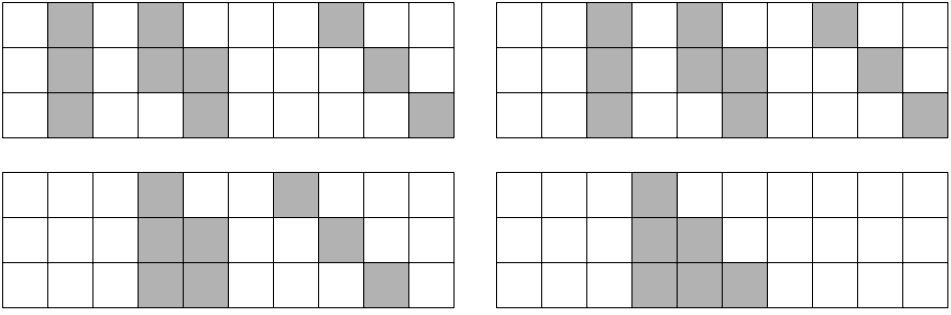

\begin{lemma}\label{pig}
Consider a shape sequence $\mathcal{S}$ of shapes $S_1,\ldots, S_k$, with elements in rows $i_s, \ldots, i_l$.  Let
$X_j^i$ be the number of elements in $S_j$ in row $i$, and let $X^i = \max_j
X_j^i$.  Then \[p(\mathcal{S}) \le \sum_{i=i_s}^{i_l} X^i.\]
\end{lemma}

\begin{proof}
Consider a region $R$ with $h(\mathcal{S})$ rows and $n$ columns colored with colors $\{1,2, \ldots, p(\mathcal{S})\}$.  Assume that every instance of $\mathcal{S}$ in this region contains every color, i.e. it is not possible to find an instance of $\mathcal{S}$ in these rows where every shape in the sequence lacks a particular color.  Thus for each color $l$, $1 \le l \le p(\mathcal{S})$, we can partition the interval $[1,n]$ into $k_l$ intervals $[1,c^l_1), [c^l_1,c^l_2), \ldots, [c^l_{k_l-1}, n]$, with the property that $k_l \le k$ and all copies of $S_j$ located at columns in the $j$th interval contain color $l$. We adopt the convention that $c^l_0=1$, $c^l_{k_l}=n$ and if $c^l_{j-1}=c^l_j $ then the interval $[c^l_{j-1},c^l_j)$ is empty. The following procedure describes how to do this for a given color $l$.

\begin{enumerate}
\item Set $\alpha =1$. 
\item If all copies of $S_\alpha$ at locations $\ge c^l_{\alpha-1}$ contain color $l$, then
\begin{itemize}
\item Set $c^l_\alpha = n$ 
\item Set $k_l = \alpha$. 
\item STOP.
\end{itemize}
\item Else
\begin{itemize}
\item Let $c^l_\alpha$ be the smallest number such that $c^l_\alpha \ge c^l_{\alpha-1}$ and the copy of $S_\alpha$ at column $c^l_\alpha$  does not contain color $l$. 
\item Increment $\alpha$ by 1.
\item Return to Step 2.
\end{itemize}
\end{enumerate}

The condition that every instance of $\mathcal{S}$ contains every color guarantees that this procedure returns a partition: If the procedure reaches a state where $\alpha=k$, then there are values $c_1^l \le c_2^l \le \cdots \le c_{k-1}^l$ such that the shape $S_j$ at location $c_j^l$ does not contain color $l$.  Thus since every instance of $\mathcal{S}$ contains $l$, every copy of $S_k$ at location $\ge c_{k-1}^l$ must contain color $l$, and the procedure will terminate in Step 2.  Let $C$ be the set $\{1, c_1^1, c_2^1, \ldots, c_{k_1}^1, c_1^2, c_2^2, \ldots, c_{k_2}^2, \ldots, c_1^{p(\mathcal{S})}, c_2^{p(\mathcal{S})}, \ldots, c_{k_{p(\mathcal{S})}}^{p(\mathcal{S})}\}$.  Relabel the elements of $C$ so that $C= \{c_1, c_2, \ldots, c_q\}$ and $c_1 \le c_2 \le \cdots \le c_q$.  Since $q \le k\cdot p(\mathcal{S}) + 1$ and we can choose $n$ as large as we want, we can find a difference $c_p-c_{p-1}$ as large as we want. Choose $n$ large enough so that this number $m = c_p-c_{p-1}$ is much bigger than $w(\mathcal{S})$. This difference corresponds to a region $R'$ with $m$ columns (the columns in the interval $[c_{p-1},c_p)$), where for each color $l$ there is a shape $S^l \in \mathcal{S}$ such that every copy of $S^l$ contains the color $l$. 

For each color $l$, and all $1 \le i \le h(\mathcal{S})$ let $l_i$ be the number of times the color
$l$ appears in the $i$th row in $R'$. Any appearance of $l$ in the $i$th row can be contained in at most $X^i$ copies of $S^l$.  There are at least
$m-w(\mathcal{S})$ copies of $S^l$ in the region, and thus \[ l_1X^1
+ l_2X^2 + \cdots + l_{h(\mathcal{S})}X^{h(\mathcal{S})} \ge
m-w(\mathcal{S}). \] 

Since there are $m$ columns in the region $R'$, $1_i + 2_i + \cdots + p(\mathcal{S})_i = m$, and if we add up the equations for each color, we get
\[ X^1m + X^2m + \cdots + X^{h(\mathcal{S})}m \ge (m-w(\mathcal{S}))p(\mathcal{S}). \] 
To finish the proof, divide both sides by $m$, and note that by making $m$ large, $(m-w(\mathcal{S}))/m$ can be as close to 1 as desired.
\end{proof}

Now to prove upper bounds on $p(G)$, we translate problems about
polychromatically coloring graphs into problems about polychromatically
coloring shape sequences.  We consider an arbitrary simple coloring of an
enormous hypercube.  Then we note that the color classes covered by an
embedding of $G$ are a shape sequence $\mathcal{S}$ in the grid of color
classes.  Further, any instance of $\mathcal{S}$ corresponds to the color
classes covered by the edges of some embedding of $G$.  Since Lemma~\ref{pig}
gives an upper bound on $p(\mathcal{S})$, we get an upper bound on $p(G)$. 

As an example of how to apply Lemma~\ref{pig}, we now prove the upper bound
on $p(Q_d)$ in Theorem~\ref{pd}. This result was originally proved in \cite{Off08}, but
this proof is more streamlined, and will provide useful preparation
for proving new results later. Define an $i \times j$ \emph{parallelogram} to be a set of color classes of the following form: $\{(a + \alpha , b + \beta): 0 \le \alpha < j, 0 \le \beta
< i \}$.  We say that a color class is at coordinate $(\alpha,\beta)$ in such a parallelogram if it is of the form $(a + \alpha , b + \beta)$.

\begin{figure}
\begin{center}
\small
\newcommand{\inS}[3]{\fill[gray!40]  (#1+#3,-#2-1) rectangle (#1+1+#3,-#2);  }
\begin{tikzpicture}[scale=1.1]
\foreach \x/\y in {0/0, 0/1, 0/2,0/3}{\inS{\x}{\y}{3}}
\foreach \x/\y in {0/0, 0/1, 0/2, 1/1,1/2,1/3}{\inS{\x}{\y}{5}}
\foreach \x/\y in {0/0, 0/1, 1/1, 1/2,2/2,2/3}{\inS{\x}{\y}{8}}
\foreach \x/\y in {0/0, 1/1, 2/2, 3/3}{\inS{\x}{\y}{10}}
  \foreach \i in {0,...,13}{
  \foreach \j in {0,...,3} {
  \pgfmathtruncatemacro{\x}{10-\i+\j} 
  \ifthenelse{\x > -1}{  \draw (\i+0.5,-\j-0.5) node{(\i,\x)} (\i,-\j-1) rectangle (\i+1,-\j);  }{}
}}  
\end{tikzpicture}
\end{center}
\caption{The shape sequence corresponding to the embedding  $[1101{*}100010{*}111{*}00101{*}]$  of $Q_4$ in $Q_{22}$.}\label{fig-q4}
\end{figure}

For an example of a shape sequence corresponding to an embedding of a graph, consider the  embedding $[1101{*}100010{*}111{*}00101{*}]$ of $Q_4$ in $Q_{22}$ 
(see Figure~\ref{fig-q4}).
Edges using the leftmost star are in color classes (3,7), (3,8), (3,9), and (3,10), a $4 \times 1$ parallelogram.
Edges using the second star from the left are in color classes (5,5), (5,6), (5,7), (6,5), (6,6), and (6,7), a $3 \times 2$ parallelogram.
Edges using the third star from the left are in color classes (8,2), (8,3), (9,2), (9,3), (10,2), and (10,3),  a $2 \times 3$ parallelogram.
Edges using the fourth star from the left are in color classes (10,0), (11,0), (12,0), and (13,0), a $1 \times 4$ parallelogram. Thus the shape sequence corresponding to $Q_4$ consists of four parallelograms, all occupying the same four rows, where each parallelogram corresponds to the edges using one of the four stars.  Further,  we can create any other instance of this shape sequence in the same four rows by rearranging some of the stars.  For example, $[1101{*}100010{*}111{*}{*}00101]$ would have the first three shapes identical, with the fourth shape shifted two columns to the left.  These observations are generalized for $Q_d$ in the following fact.

\begin{fact}\label{shseqqd}
Let $n \ge d \ge 1$. Every shape sequence for an embedding of $Q_d$ in $Q_n$ consists of $d$ shapes $S_1, \ldots, S_d$ where $S_i$ is a $(d-i+1) \times i$ parallelogram, and each shape occupies the same $d$ rows. The color classes in $S_i$ correspond to the edges using the $i$th star from the left.  Conversely, every instance of such a shape sequence corresponds to some embedding of $Q_d$ in $Q_n$.
\end{fact}

See Figures~\ref{instances}, \ref{4cube}, and~\ref{5cube} for examples of shape sequences corresponding to $Q_3$, $Q_4$, and $Q_5$, respectively.

\begin{proofof}{Theorem~\ref{pd}, upper bound}
By Lemma~\ref{simple}, we may consider a simple $Q_d$-polychromatic $p(Q_d)$-coloring on an arbitrarily large hypercube. Fact 6 describes the shape sequence for $Q_d$. For any of the parallelograms in the shape sequence, the maximum number of color classes in the $i$th row is $\min\{i, d-i+1\}$. Thus, using the notation of Lemma~\ref{pig}, in the shape sequence for $Q_d$, $X^i = \max_j X_j^i = \min\{i, d-i+1\}$. Applying
Lemma~\ref{pig}, we get $p(Q_d) \le 1 + 2 + \cdots + \lceil d/2 \rceil +
\cdots + 2 + 1 = (d+1)^2/4$ if $d$ is odd, and $d(d+2)/4$ if $d$ is
even. 
\end{proofof}

\begin{figure}
\begin{center}
\newcommand{\inS}[3]{\fill[gray!60]  (#1+#3,-#2-1) rectangle (#1+1+#3,-#2);  }
\begin{tikzpicture}[scale=0.6]
\foreach \x/\y in {0/0, 0/1, 0/2,0/3}{\inS{\x}{\y}{2}}
\foreach \x/\y in {0/0, 0/1, 0/2,1/1, 1/2,1/3}{\inS{\x}{\y}{5}}
\foreach \x/\y in {0/0, 0/1, 1/1, 1/2,2/2,2/3}{\inS{\x}{\y}{10}}
\foreach \x/\y in {0/0, 1/1, 2/2,3/3}{\inS{\x}{\y}{14}}
\draw (0,0) grid (21,-4);
\end{tikzpicture}
\end{center}
\caption{A shape sequence for $Q_4$.}\label{4cube}
\end{figure}
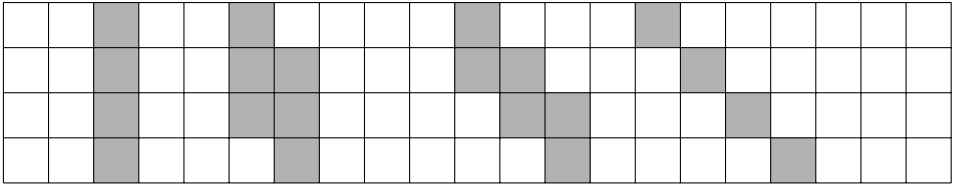

\begin{figure}
\begin{center}
\newcommand{\inS}[3]{\fill[gray!60]  (#1+#3,-#2-1) rectangle (#1+1+#3,-#2);  }
\begin{tikzpicture}[scale=0.6]
\foreach \x/\y in {0/0, 0/1, 0/2,0/3,0/4}{\inS{\x}{\y}{1}}
\foreach \x/\y in {0/0, 0/1, 0/2,0/3,1/1, 1/2,1/3,1/4}{\inS{\x}{\y}{3}}
\foreach \x/\y in {0/0, 0/1,0/2, 1/1, 1/2,1/3,2/2,2/3,2/4}{\inS{\x}{\y}{8}}
\foreach \x/\y in {0/0, 0/1, 1/1, 1/2,2/2,2/3,3/3,3/4}{\inS{\x}{\y}{12}}
\foreach \x/\y in {0/0, 1/1, 2/2,3/3,4/4}{\inS{\x}{\y}{16}}
\draw (0,0) grid (21,-5);
\end{tikzpicture}
\end{center}
\caption{A shape sequence for $Q_5$.}\label{5cube}
\end{figure}
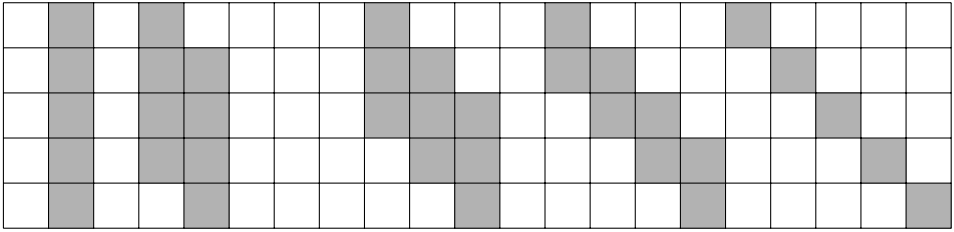

We now turn our attention to lower bounds. To prove lower bounds on $p(G)$, we explicitly describe a simple coloring of the hypercube by assigning a particular color to each color
class. Then we analyze what color classes must be contained in the shape sequence of any embedding of the graph $G$, and show it contains all colors. For instance, here is a proof of the lower bound for Theorems~\ref{AKS} and \ref{pd}:

\begin{proofof}{Theorems~\ref{AKS} and \ref{pd}, lower bound}\label{pot1}
Consider the simple coloring $\chi$ where $\chi(e) =\lceil \frac{d+1}{2} \rceil \cdot l(e) + r(e) \pmod{q}$, where $q=\frac{(d+1)^2}{4}$ if $d$ is odd, and $q= \frac{d(d+2)}{4}$ if $d$ is even. Then the $\lceil \frac{d}{2} \rceil $th shape in the shape sequence for $Q_d$ is a $\lceil \frac{d+1}{2} \rceil \times \lceil \frac{d}{2} \rceil$ parallelogram, and thus contains $q$ color classes.  The elements in each column in this parallelogram contain $\lceil \frac{d+1}{2} \rceil$ consecutive colors $\pmod{q}$, and no two columns of the parallelogram share any colors, so this shape contains all $q$ colors regardless of its position in the grid of color classes.  Thus any shape sequence corresponding to $Q_d$ contains all colors.
\end{proofof}

\section{Polychromatic Numbers for Punctured Cubes}\label{presults}

We now use the techniques of Section~\ref{boupg} to
determine $p(G)$ for other subgraphs $G$ of the hypercube besides subcubes.  In
each case, we explicitly describe a simple coloring to prove a lower bound on $p(G)$, and
examine possible shape sequences corresponding to $G$ and apply
Lemma~\ref{pig} to prove an upper bound. 

We consider graphs which are obtained by deleting a vertex or edge from a hypercube.  We call such graphs \textit{punctured hypercubes} or \textit{punctured cubes} for short.  Let $Q_d \setminus v$ and $Q_d \setminus e$ denote the graphs obtained by deleting one
vertex or one edge from $Q_d$, respectively. If $d$ is odd we call these \textit{odd punctured cubes}, and if $d$ is even we call them \textit{even punctured cubes}. When embedding a punctured cube in a larger hypercube, we can think of generating
each embedding by first embedding $Q_d$, then deleting a given
vertex or edge from the embedded graph (not from the base graph of
course).  The shape sequence corresponding to such an embedding depends on
which vertex or edge is deleted.  Since an embedding of $Q_d$ in
$Q_n$ can be represented as an $n$-bit vector with $d$ stars, we
say that deleting the vertex corresponding to some $d$-bit string 
corresponds to deleting from the embedding of $Q_d$ that vertex in
$Q_n$ which replaces the $d$ stars with that given $d$-bit string.  We
also use the same notion for edges.  For example, we might consider the
subgraph $Q_3 \setminus v$ in $Q_6$ corresponding to the embedding $[01{*}0{*}{*}]$ of
$Q_3$ with the vertex corresponding to [011]
deleted, so the embedded subgraph will not contain the vertex
[010011] or any incident edges.

\begin{proposition}\label{submono}
If $G$ is a subgraph of $H$, and both are subgraphs of the hypercube, then $p(G) \le p(H)$.  
\end{proposition}

\begin{corollary}\label{puncmon}
For any $d \ge 2$,
\[p(Q_{d-1}) \le p(Q_d\setminus v) \le p(Q_d\setminus e) \le p(Q_d).\]
\end{corollary}

Letting $d=2k$ when $d$ is even, or $d=2k-1$ when $d$ is odd, Theorem~\ref{pd} states that $p(Q_{2k-1}) = k^2$ and $p(Q_{2k}) = k^2+k$.  Thus Corollary~\ref{puncmon} implies for odd punctured cubes 

\[k^2-k = p(Q_{2k-2}) \le p(Q_{2k-1}\setminus v) \le p(Q_{2k-1}\setminus e) \le p(Q_{2k-1}) = k^2,\]

and for even punctured cubes

\[k^2 = p(Q_{2k-1}) \le p(Q_{2k}\setminus v) \le p(Q_{2k}\setminus e) \le p(Q_{2k}) = k^2+k. \]

In the case of odd punctured cubes, we determine the polychromatic number exactly. In the case of even punctured cubes, we give upper and lower bounds but, with a few exceptions, we do not know the exact polychromatic number.

\begin{lemma}\label{ub-v}
For all $d\ge 2$, $p(Q_d \setminus v)  \le p(Q_d) - 1$.
\end{lemma}
\begin{proof}
Recall that the shape sequence corresponding to  $Q_d$ is a sequence of $d$ parallelograms as described in Fact~\ref{shseqqd}.  If the vertex corresponding to $[11 \ldots 1]$ is deleted from $Q_d$, then the shape sequence corresponding to this embedding is identical to the shape sequence for $Q_d$ except the single element in the $d$th row is deleted from each parallelogram. For example, for $Q_4 \setminus v$, the resulting shape sequence is shown in Figure~\ref{4mv}.  Since the rest of the rows are not altered, the result follows from  Lemma~\ref{pig}.
\end{proof}

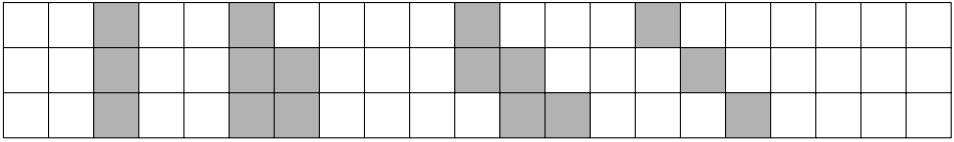
\begin{figure}
\begin{center}
\newcommand{\inS}[3]{\fill[gray!60]  (#1+#3,-#2-1) rectangle (#1+1+#3,-#2);  }
\begin{tikzpicture}[scale=0.6]
\foreach \x/\y in {0/0, 0/1, 0/2}{\inS{\x}{\y}{2}}
\foreach \x/\y in {0/0, 0/1, 0/2,1/1, 1/2}{\inS{\x}{\y}{5}}
\foreach \x/\y in {0/0, 0/1, 1/1, 1/2,2/2}{\inS{\x}{\y}{10}}
\foreach \x/\y in {0/0, 1/1, 2/2}{\inS{\x}{\y}{14}}
\draw (0,0) grid (21,-3);
\end{tikzpicture}
\end{center}
\caption{A shape sequence for $Q_4\setminus v$ where $v$ corresponds to $[1111]$.}\label{4mv}
\end{figure}

\begin{theorem}\label{qmv}
For all $k \ge 2$, $p(Q_{2k-1} \setminus v) = p(Q_{2k-1} \setminus e) = k^2-1$.
\end{theorem}
\begin{proof}
By Corollary~\ref{puncmon} it suffices to show that $k^2-1 \le p(Q_{2k-1} \setminus v)$ and $p(Q_{2k-1} \setminus e) \le k^2-1$.

To prove the first inequality, we show that the simple coloring $\chi: E(Q_n) \rightarrow [k^2-1]$ where $\chi(e) = k \cdot l(e) + r(e) \pmod{k^2 - 1}$ is $(Q_{2k-1}\setminus v)$-polychromatic.  

The $k$th shape in the shape sequence for $Q_{2k-1}$ is a $k \times k$ parallelogram. Since each column in this parallelogram contains $k$ consecutive colors $\pmod{k^2-1}$, and only the first and last columns in the parallelogram share any colors (the color at coordinate (0,0) and the color at coordinate $(k-1,k-1)$  are the same), this shape contains all $k^2-1$ colors. Since all the edges in the color classes in this parallelogram form a matching, deleting a vertex can remove at most one edge in any of these color classes. Every color class except the ones at coordinates (0,0), $(k-1,k-1)$, $(0,k-1)$, and $(k-1,0)$ contains at least two edges, and since the colors at coordinates (0,0) and $(k-1,k-1)$  are the same, there are only four vertices which can be deleted from $Q_{2k-1}$ where in the corresponding shape sequence for $Q_{2k-1} \setminus v$ the $k$th shape will not contain all colors.  In these vertices, the first $k-1$ entries are all 0's or all 1's, and the last $k-1$ entries are all 1's or all 0's, opposite from the first $k-1$ entries.  The $k$th entry can be 1 or a 0.  For instance, if $k=3$, the vertices in question correspond to
$[00011]$, $[11100]$, $[00111]$, and $[11000]$.  

Since $k(l+1) + (r-1) - (kl+r) = k-1$, and since $(k-1)$ divides $k^2-1$, the coloring $\chi(e) = k \cdot l(e) + r(e) \pmod{k^2 - 1}$ has the property that there are exactly $k+1$ colors in each row, each congruent $\pmod{k-1}$. Further, consecutive rows have colors in consecutive congruence classes. See Figures~\ref{11100} and~\ref{11000} for examples where $k=3$.

Consider an embedding of $Q_{2k-1} \setminus v$ and assume $v$ is one of the four vertices where the $k$th shape does not contain all colors. This choice of $v$ deletes one color class from the $k$th row of the $k$th shape. In the shape sequence for $Q_{2k-1} \setminus v$ the colors in the $k$th row are also found only in the first and $(2k-1)$th row. Without loss of generality assume these colors are congruent to 0 $\pmod{k-1}$.

In the two cases where the $k$th entry of $v$ is the same as the entries before it, the $(k-1)$st shape will remain intact (a $(k+1) \times (k-1)$ parallelogram). In the two cases where the $k$th entry of $v$ is the same as the entries after it, the $(k+1)$st shape will remain intact (a $(k-1) \times (k+1)$ parallelogram). Thus it suffices to show that both of these parallelograms contain all colors congruent to 0 $\pmod{k-1}$. In the $(k-1)$st shape, a $(k+1) \times (k-1)$ parallelogram, the colors at coordinates (0,0), $(0,k-1)$, $(1,k-2)$, $\ldots$, $(k-2,1)$, $(k-2, k)$ are $(k+1)$ consecutive multiples of $k-1$ $\pmod{k^2-1}$.  In the $(k+1)$st shape, the colors at coordinates (0,0), $(k, k-2)$, $(1,k-2)$, $(2,k-3)$, $\ldots$, $(k-1,0)$ are $(k+1)$ consecutive multiples of $k-1$ $\pmod{k^2-1}$.  Hence the coloring $\chi$ is $(Q_{2k-1}\setminus v)$-polychromatic, as claimed.

We now prove the inequality $p(Q_{2k-1} \setminus e) \le k^2-1$. If the edge corresponding to the edge with a star in the $k$th position, 0's to the left and 1's to the right (e.g. for $k=3$, $e=
[00{*}11]$) is deleted, then a single color class in the $k$th row of the $k$th
shape is deleted from the shape sequence for $Q_{2k-1}$ (Figure~\ref{5me} shows the shape sequence for $k=3$).  This reduces the maximum width of the shapes in the $k$th row by one, leaving the other rows intact, and the inequality follows from Lemma~\ref{pig}.
\end{proof}

\begin{figure}
\begin{center}
\junk{
\newcommand{\inS}[3]{\fill[gray!60]  (#1+#3,-#2-1) rectangle (#1+1+#3,-#2);  }
\begin{tikzpicture}[scale=0.6]
\foreach \x/\y in {0/0, 0/1, 0/2,0/3,0/4}{\inS{\x}{\y}{1}}
\foreach \x/\y in {0/0, 0/1, 0/2,0/3,1/1, 1/2,1/3,1/4}{\inS{\x}{\y}{3}}
\foreach \x/\y in {0/0, 0/1,0/2, 1/1, 1/2,1/3, 2/3,2/4}{\inS{\x}{\y}{8}}
\foreach \x/\y in {0/0, 0/1, 1/1, 1/2,2/2,2/3, 3/4}{\inS{\x}{\y}{12}}
\foreach \x/\y in {0/0, 1/1, 2/2,3/3,4/4}{\inS{\x}{\y}{16}}
\draw (0,0) grid (21,-5);
\end{tikzpicture}
\vskip 1em
}
\newcommand{\inS}[3]{\fill[gray!40]  (#1+#3,-#2-1) rectangle (#1+1+#3,-#2);  }
\begin{tikzpicture}[scale=0.6]
\foreach \x/\y in {0/0, 0/1, 0/2,0/3,0/4}{\inS{\x}{\y}{1}}
\foreach \x/\y in {0/0, 0/1, 0/2,0/3,1/1, 1/2,1/3,1/4}{\inS{\x}{\y}{3}}
\foreach \x/\y in {0/0, 0/1,0/2, 1/1, 1/2,1/3, 2/3,2/4}{\inS{\x}{\y}{8}}
\foreach \x/\y in {0/0, 0/1, 1/1, 1/2, 2/2,2/3, 3/4}{\inS{\x}{\y}{12}}
\foreach \x/\y in {0/0, 1/1, 2/2,3/3,4/4}{\inS{\x}{\y}{16}}
\draw (0,0) grid (21,-5);
  \foreach \i in {0,...,20}{
  \foreach \j in {0,...,4} {
  \pgfmathtruncatemacro{\x}{mod(\j+2*\i,8)} 
  \draw (\i+0.5,-\j-0.5) node{\x};  
}}  
\end{tikzpicture}
\end{center}

\caption{A shape sequence for $Q_5\setminus v$, where $v$ corresponds to $[11100]$, with the coloring $\chi(e) = 3 \cdot l(e) + r(e) \pmod{8}$.}\label{11100}
\end{figure}

\begin{figure}
\begin{center}
\junk{
\newcommand{\inS}[3]{\fill[gray!60]  (#1+#3,-#2-1) rectangle (#1+1+#3,-#2);  }
\begin{tikzpicture}[scale=0.6]
\foreach \x/\y in {0/0, 0/1, 0/2,0/3,0/4}{\inS{\x}{\y}{1}}
\foreach \x/\y in {0/0, 0/1, 0/2,0/3, 1/2,1/3,1/4}{\inS{\x}{\y}{3}}
\foreach \x/\y in {0/0, 0/1,0/2, 1/1, 1/2,1/3, 2/3,2/4}{\inS{\x}{\y}{8}}
\foreach \x/\y in {0/0, 0/1, 1/1, 1/2,2/2,2/3, 3/3,3/4}{\inS{\x}{\y}{12}}
\foreach \x/\y in {0/0, 1/1, 2/2,3/3,4/4}{\inS{\x}{\y}{16}}
\draw (0,0) grid (21,-5);
\end{tikzpicture}
\vskip 1em
}
\newcommand{\inS}[3]{\fill[gray!40]  (#1+#3,-#2-1) rectangle (#1+1+#3,-#2);  }
\begin{tikzpicture}[scale=0.6]
\foreach \x/\y in {0/0, 0/1, 0/2,0/3,0/4}{\inS{\x}{\y}{1}}
\foreach \x/\y in {0/0, 0/1, 0/2,0/3, 1/2,1/3,1/4}{\inS{\x}{\y}{3}}
\foreach \x/\y in {0/0, 0/1,0/2, 1/1, 1/2,1/3, 2/3,2/4}{\inS{\x}{\y}{8}}
\foreach \x/\y in {0/0, 0/1, 1/1, 1/2, 2/2,2/3, 3/3,3/4}{\inS{\x}{\y}{12}}
\foreach \x/\y in {0/0, 1/1, 2/2,3/3,4/4}{\inS{\x}{\y}{16}}
\draw (0,0) grid (21,-5);
  \foreach \i in {0,...,20}{
  \foreach \j in {0,...,4} {
  \pgfmathtruncatemacro{\x}{mod(\j+2*\i,8)} 
  \draw (\i+0.5,-\j-0.5) node{\x};  
}}  
\end{tikzpicture}
\end{center}
\caption{A shape sequence for $Q_5\setminus v$, where $v$ corresponds to $[11000]$, with the coloring $\chi(e) = 3 \cdot l(e) + r(e) \pmod{8}$.}\label{11000}
\end{figure}


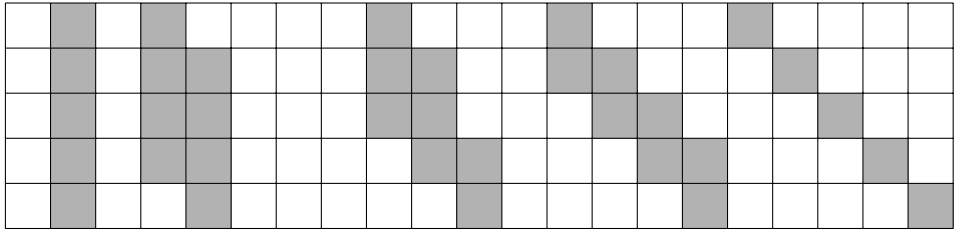
\begin{figure}
\begin{center}
\newcommand{\inS}[3]{\fill[gray!60]  (#1+#3,-#2-1) rectangle (#1+1+#3,-#2);  }
\begin{tikzpicture}[scale=0.6]
\foreach \x/\y in {0/0, 0/1, 0/2,0/3,0/4}{\inS{\x}{\y}{1}}
\foreach \x/\y in {0/0, 0/1, 0/2,0/3,1/1, 1/2,1/3,1/4}{\inS{\x}{\y}{3}}
\foreach \x/\y in {0/0, 0/1,0/2, 1/1, 1/2,1/3, 2/3,2/4}{\inS{\x}{\y}{8}}
\foreach \x/\y in {0/0, 0/1, 1/1, 1/2,2/2,2/3,3/3, 3/4}{\inS{\x}{\y}{12}}
\foreach \x/\y in {0/0, 1/1, 2/2,3/3,4/4}{\inS{\x}{\y}{16}}
\draw (0,0) grid (21,-5);
\end{tikzpicture}
\end{center}
\caption{A shape sequence for $Q_5\setminus e$, where $e$ corresponds to $[11{*}00]$.}\label{5me}
\end{figure}

We now turn our attention to even punctured cubes, first giving exact values for $p(Q_4 \setminus v)$, $p(Q_4 \setminus e)$, and $p(Q_6 \setminus e)$.

\begin{theorem}\label{p4mv}
$p(Q_4 \setminus v)  = 5$.
\end{theorem}
\begin{proof}
Since $p(Q_4) = 6$, the upper bound follows from Lemma~\ref{ub-v}.

For the lower bound, we show that the simple coloring $\chi: E(Q_n) \rightarrow [5]$ where $\chi(e) = 3 \cdot l(e) + r(e) \pmod{5}$ is $(Q_4\setminus v)$-polychromatic (see Figure~\ref{chi4mv}). To see this, note that the shape sequence for $Q_4$ contains all five colors in both the second shape and the third shape regardless of where those shapes appear. If a vertex other than one corresponding to $[0000]$, $[1111]$, $[0011]$, or $[1100]$ is deleted, then the shape sequence corresponding to $Q_4 \setminus v$ will have the second or third shape still intact, and thus be polychromatic.  If
either of the first two of these are deleted, then even though the second
shape will lose one element, it will still contain all five colors, and if
either of the last two are deleted, the same holds for the third
shape. 
\end{proof}

\begin{figure}
\begin{center}
\newcommand{\inS}[3]{\fill[gray!40,thick]  (#1+#3,-#2-1) rectangle (#1+1+#3,-#2);  }
\begin{tikzpicture}[scale=0.65]
\foreach \x/\y in {0/0, 0/1, 0/2, 0/3}{\inS{\x}{\y}{1}}
\foreach \x/\y in {0/0, 0/1, 1/1, 0/2, 1/2, 1/3}{\inS{\x}{\y}{3}}
\foreach \x/\y in {0/0, 0/1,  1/1, 1/2,  2/2, 2/3}{\inS{\x}{\y}{8}}
\foreach \x/\y in {0/0, 1/1,   2/2,  3/3}{\inS{\x}{\y}{12}}
\draw (0,0) grid (18,-4);
  \foreach \i in {0,...,17}{
  \foreach \j in {0,...,3} {
  \pgfmathtruncatemacro{\x}{mod(\j+2*\i,5)} 
  \draw (\i+0.5,-\j-0.5) node{\x};  
}}  
\end{tikzpicture}
\end{center}
\caption{A shape sequence for $Q_4$ with the coloring $\chi(e) = 3 \cdot l(e) + r(e) \pmod{5}$.}\label{chi4mv}
\end{figure}

\begin{theorem}\label{p4me}
$p(Q_{4} \setminus e) =6$.
\end{theorem}

\begin{proof}
Since $p(Q_4)=6$, the upper bound follows from Corollary~\ref{puncmon}.

For the lower bound, we show that the simple coloring $\chi: E(Q_n) \rightarrow [6]$ where $\chi(e) = 4 \cdot l(e) + r(e) \pmod{6}$ is $(Q_4\setminus e)$-polychromatic (see Figure~\ref{chi4me}).  Since deleting an edge affects at most one color class, it suffices to show that each color is present in two shapes of the shape sequence for $Q_4$. For this coloring, each row contains only two colors, which are congruent $\pmod{3}$, and adjacent rows contain consecutive congruence classes. In the four rows containing a shape sequence for $Q_4$, the two colors in the first and fourth rows are in the first and third shapes, and the colors in the two middle rows are in the second and third shapes.
\end{proof}

\begin{figure}
\begin{center}
\newcommand{\inS}[3]{\fill[gray!40,thick]  (#1+#3,-#2-1) rectangle (#1+1+#3,-#2);  }
\begin{tikzpicture}[scale=0.65]
\foreach \x/\y in {0/0, 0/1, 0/2, 0/3}{\inS{\x}{\y}{1}}
\foreach \x/\y in {0/0, 0/1, 1/1, 0/2, 1/2, 1/3}{\inS{\x}{\y}{3}}
\foreach \x/\y in {0/0, 0/1,  1/1, 1/2,  2/2, 2/3}{\inS{\x}{\y}{8}}
\foreach \x/\y in {0/0, 1/1,   2/2,  3/3}{\inS{\x}{\y}{12}}
\draw (0,0) grid (18,-4);
  \foreach \i in {0,...,17}{
  \foreach \j in {0,...,3} {
  \pgfmathtruncatemacro{\x}{mod(\j+3*\i,6)} 
  \draw (\i+0.5,-\j-0.5) node{\x};  
}}  
\end{tikzpicture}
\end{center}
\caption{A shape sequence for $Q_4$ with the coloring $\chi(e) = 4 \cdot l(e) + r(e) \pmod{6}$.}\label{chi4me}
\end{figure}

\begin{theorem}\label{p6me}
$p(Q_{6} \setminus e) =12$.
\end{theorem}

\begin{proof}
Since $p(Q_6)=12$, the upper bound follows from Corollary~\ref{puncmon}.

For the lower bound, we show that the simple coloring $\chi: E(Q_n) \rightarrow [12]$ where $\chi(e) = 5 \cdot l(e) + r(e) \pmod{12}$ is $(Q_6\setminus e)$-polychromatic (see Figure~\ref{chi6me}).  Since deleting an edge affects at most one color class, it suffices to show that each color is present in two shapes of the shape sequence for $Q_6$. For this coloring, each row contains only three colors, which are congruent $\pmod{4}$, and adjacent rows contain consecutive congruence classes. In the six rows containing a shape sequence for $Q_6$, the six colors in the first, second, fifth and sixth rows are in the second and fifth shapes, while the other six colors in the two middle rows are in the third and fourth shapes.
\end{proof}

\begin{figure}
\begin{center}
\newcommand{\inS}[3]{\fill[gray!40,thick]  (#1+#3,-#2-1) rectangle (#1+1+#3,-#2);  }
\begin{tikzpicture}[scale=0.65]
\foreach \x/\y in {0/0, 0/1, 0/2, 0/3, 0/4, 1/1, 1/2, 1/3, 1/4, 1/5}{\inS{\x}{\y}{1}}
\foreach \x/\y in {0/0, 0/1, 1/1, 0/2, 0/3, 1/2, 1/3, 1/4, 2/2, 2/3, 2/4, 2/5}{\inS{\x}{\y}{4}}
\foreach \x/\y in {0/0, 0/1, 0/2, 1/1, 1/2, 1/3, 2/2, 2/3, 2/4, 3/3, 3/4, 3/5}{\inS{\x}{\y}{8}}
\foreach \x/\y in {0/0, 0/1, 1/1, 1/2,  2/2, 2/3,  3/3, 3/4, 4/4, 4/5}{\inS{\x}{\y}{12}}
\draw (0,0) grid (18,-6);
  \foreach \i in {0,...,17}{
  \foreach \j in {0,...,5} {
  \pgfmathtruncatemacro{\x}{mod(\j+4*\i,12)} 
  \draw (\i+0.5,-\j-0.5) node{\x};  
}}  
\end{tikzpicture}
\end{center}
\caption{The second, third, fourth, and fifth shapes in the shape sequence for $Q_6$ with the coloring $\chi(e) = 5 \cdot l(e) + r(e) \pmod{12}$.}\label{chi6me}
\end{figure}

\begin{theorem}\label{pq2kmv}
For all $k \ge 3$, $p(Q_{2k} \setminus v) \ge k^2+k-2 = (k-1)(k+2)$.
\end{theorem}
\begin{proof}
We show that the simple coloring $\chi: E(Q_n) \rightarrow [(k-1)(k+2)]$ where $\chi(e) = k \cdot l(e) + r(e) \pmod{(k-1)(k+2)}$ is $(Q_4\setminus v)$-polychromatic. In this coloring, each row only contains $k+2$ colors, which are all congruent $\pmod{k-1}$. Adjacent rows contain colors in consecutive congruence classes.

First note that a $k \times (k+1)$ parallelogram, the $(k+1)$st shape in the shape sequence for $Q_{2k}$, contains all colors.   All color classes in this parallelogram except for the ones at coordinates $(0,0)$, $(k,0)$, $(0,k-1)$, and $(k,k-1)$ have at least two parallel edges, and thus deleting a vertex from $Q_{2k}$ will leave at least one edge in each of these color classes. Further, the edges in the classes at coordinates $(0,0)$ and $(k,k-1)$ have the same color.  Thus the only four vertices whose deletion will cause this shape to lose a color are those with first $k$ coordinates all 0 or all 1, and last $k-1$ coordinates all 1 or all 0, opposite from the first.  For one of these vertices, if the first $k$ coordinates are all 1, the color class at coordinate $(k,0)$ is deleted.  If the first $k$ coordinates are all 0, the color class at coordinate $(0,k-1)$ is deleted. Without loss of generality, assume that the color classes at coordinates $(0,k-1)$ and $(k,0)$ are in rows containing colors congruent to 0 and 1 $\pmod{k-1}$ respectively. See Figure~\ref{11110000} for an example with $k=4$ and $v=[11110000]$.

In either of these cases, the $k$th shape in the shape sequence for $Q_{2k}\setminus v$ must contain the deleted color.  In the $k$th shape, the color classes at coordinates $(0,0)$, $(0,k-1)$, $(1,k-2)$, $\ldots$, $(k-2,1)$, $(k-2,k)$, and $(k-1,k-1)$ contain $k+2$ consecutive multiples of $k-1$ $\pmod{(k-1)(k+2)}$ and thus contain all colors congruent to 0 $\pmod{k-1}$.  Similarly, the color classes at coordinates $(0,1)$, $(1,0)$, $(1,k-1)$, $(2,k-2)$, $\ldots$, $(k-1,1)$, and $(k-1,k)$, contain $k+2$ consecutive numbers $\pmod{(k-1)(k+2)}$ that are congruent to 1 $\pmod{k-1}$ and thus contain all colors congruent to 1 $\pmod{k-1}$. None of these color classes can be deleted by deleting any of the four vertices under consideration.
\end{proof}

\begin{figure}
\begin{center}
\newcommand{\inS}[3]{\fill[gray!40]  (#1+#3,-#2-1) rectangle (#1+1+#3,-#2);  }
\begin{tikzpicture}[scale=0.6]
\foreach \x/\y in {0/0, 0/1, 0/2,0/3,0/4,1/1,1/2,1/3,1/4,1/5,2/2,2/3,2/4,2/5,2/6,3/4,3/5,3/6,3/7}{\inS{\x}{\y}{3}}
\foreach \x/\y in {0/0,0/1,0/2,0/3, 1/1,1/2,1/3,1/4,2/2,2/3,2/4,2/5,3/3,3/4,3/5,3/6,4/5,4/6,4/7}{\inS{\x}{\y}{12}}

\draw (0,0) grid (21,-8);
  \foreach \i in {0,...,20}{
  \foreach \j in {0,...,7} {
  \pgfmathtruncatemacro{\x}{mod(\j+3*\i,18)} 
  \draw (\i+0.5,-\j-0.5) node{\x};  
}}  
\end{tikzpicture}
\end{center}

\caption{The fourth and fifth shapes in the shape sequence for $Q_8\setminus v$, where $v$ corresponds to $[11110000]$, with the coloring $\chi(e) = 4 \cdot l(e) + r(e) \pmod{18}$.}\label{11110000}
\end{figure}

\begin{corollary}\label{evensum}
For all $k\ge 3$, 
\[k^2+k-2 \le p(Q_{2k} \setminus v) \le k^2+k -1.\]
For all $k \ge 4$,
\[k^2+k-2 \le p(Q_{2k} \setminus e) \le k^2+k.\]
\end{corollary}


\subsection{Open Problems}

We conclude with some directions for future research on polychromatic edge colorings.

\begin{problem}\label{otherG}
For which other graphs $G$ can we determine $p(G)$?  
\end{problem}

A first step might be to determine the polychromatic numbers for even punctured cubes, starting with $p(Q_6 \setminus v)$ and $p(Q_8 \setminus e)$. If these polychromatic numbers are smaller than the upper bounds given in Corollary~\ref{evensum}, it would provide our first example of a polychromatic number not equal to the bound given by Lemma~\ref{pig}.  In turn, such a result might give insight into how Lemma~\ref{pig} might be strengthened--for example, its proof only uses horizontal translations, but shape sequences can be translated in any direction--or whether other ideas are necessary.

\begin{problem}\label{all}
 For each $r \ge 2$, is there some $G$ such that $p(G) = r$?  
\end{problem}

The smallest open case for Problem~\ref{all} is $r=7$.  It would be interesting even to show that the gaps between polychromatic numbers are bounded.  It might also be interesting to investigate whether (and under what circumstances) polychromatic numbers can ``jump.'' In other words, if an edge is deleted from a graph $G$, by how much can the polychromatic number of the resulting graph differ from $p(G)$?

We remark that in this paper, we follow the definition of embedding given in Alon, Krech, and Szab\'o  \cite{AKS07}, where a subgraph $H$ of the hypercube is considered an embedding of $G$ if there is a graph isomorphism from $H$ to $G$. For the purpose of this discussion we refer to this type of embedding as a \textit{subgraph embedding}. When considering Problems~\ref{otherG} and \ref{all}, it may also be interesting to consider a more restrictive definition of embedding that we call \textit{isometric embedding}. In this definition, we first choose a fixed embedding of $G$ in a hypercube $Q_m$. Then the embeddings of $G$ in some larger hypercube $Q_n$ are exactly the images of $G$ when $Q_m$ is embedded in $Q_n$. 

For all subgraphs of the hypercube considered in this paper (hypercubes and punctured cubes of dimension 3 or greater) the two types of embeddings yield the same set of subgraphs.  However in general this is not true.  For example, any path with three edges is considered a subgraph embedding of  $Q_2 \setminus e$, while only those paths whose endpoints are at distance one are considered isometric embeddings of $Q_2 \setminus e$. Considering only isometric embeddings rather than subgraph embeddings would change the polychromatic number for many graphs.  For example, using subgraph embeddings, $p(Q_2 \setminus e ) = 1$, whereas with isometric embeddings, where every embedding of $Q_2 \setminus e$ must contain 3 edges of a 4-cycle, $p(Q_2 \setminus e) = 2$ (just color the edges according to their level$\pmod{2}$).

We do not know the answer to Problem~\ref{all} in the case of subgraph embeddings, but in the case of isometric embeddings the answer is yes:  Fix $r \ge 1$, and let $G$ be the graph consisting of all edges that use the first star in a copy of $Q_r$ (i.e. $G$ is a perfect matching in $Q_r$ with all edges parallel).  Then every isometric embedding of $G$ contains edges on $r$ consecutive levels, so by coloring the edges according to their level$\pmod{r}$ we obtain a $G$-polychromatic $r$-coloring, and $p(G) \ge r$.  However the embedding of $G$ in $[{*}\ldots{*}00\ldots 0]$ where all edges use the first star covers only $r$ color classes in a simple coloring (the classes $(0,0), (0,1), \ldots, (0,r-1)$) and so $p(G) \le r$. 

Finally, Bialostocki \cite{Bia83} proved that any subgraph of the hypercube not containing $Q_2$ as a subgraph and intersecting every $Q_2$ has at most $(n+\sqrt{n})2^{n-2}$ edges. This implies for large $n$, every $Q_2$-polychromatic 2-coloring has approximately half the edges of each color. Is it possible to generalize Bialostocki's theorem? 
\begin{problem}
For large $n$, given a $G$-polychromatic coloring of $Q_n$ with $p(G)$ colors, is it true that the proportion of edges in each color class must approach $1/p(G)$?
\end{problem}

\section{Coloring Subcubes of Higher Dimension}\label{generalSection}
Alon, Krech, and Szab\'o  \cite{AKS07} suggested a generalization of the problem of finding polychromatic numbers on the hypercube where instead of edges (which themselves can be thought of as one-dimensional hypercubes), subcubes of a fixed dimension of the hypercube are colored. Let $p^i(G)$ be the polychromatic number of the graph $G$ if $Q_i$'s are colored, i.e. $p^i(G)$ is the largest number of colors with which it is possible to color the $Q_i$'s in any hypercube so that every embedding of $G$ contains a $Q_i$ of every color. To determine these polychromatic numbers, many of the ideas from edge colorings can be directly generalized.  However in Theorems~\ref{p233} and \ref{p24} we show that not everything about edge coloring generalizes in a straightforward way.

When $Q_i$'s are colored, a simple coloring is one where the color of a $Q_i$ is determined by the vector of length $(i+1)$ where the first coordinate is the number of 1's to the left of the first star, the $(i+1)$st coordinate is the number of 1's to the right of the $i$th star, and for $1<j<i+1$, the $j$th coordinate is the number of 1's between the $(j-1)$st and $j$th stars. With this definition, a proof almost identical to that of Lemma~\ref{simple} (see \"Ozkahya and Stanton \cite{OS11}) gives the following generalization.

\begin{lemma}\label{simplepi}
Let $k \ge i \ge 1$ and $G$ be a subgraph of $Q_k$. If $p^i(G) = r$, then there is a simple $G$-polychromatic $r$-coloring of the $Q_i$'s in $Q_k$.
\end{lemma}

Thus we restrict our attention to simple colorings, and consider color classes in an $(i+1)$-dimensional grid.  As with edge colorings, we refer to all $Q_i$'s with the same vector in a simple coloring as a color class. For example, the embedding $[01110{*}0{*}11{*}01001{*}11011]$ of $Q_4$ in $Q_{22}$ would be in color class $(3,0,2,2,4)$.  Define two color classes in a $(i+1)$-dimensional grid to be on the same \textit{level} if their entries have the same sum.  Define a $j_1 \times j_2 \times \cdots \times j_l$ \emph{parallelepiped} to be a set of color classes of the following form: $\{(a_1 + \alpha_1 , a_2 + \alpha_2, \ldots, a_l+\alpha_l): 0 \le \alpha_k < j_k\}$. For the remainder of the section, we restrict our attention to the case where $G=Q_d$. Shape sequences for $Q_d$ are characterized by the following generalization of Fact~\ref{shseqqd}.

\begin{fact}\label{shseqqdpi}
Let $n \ge d \ge i \ge 1$. Every shape sequence for an embedding of $Q_d$ in $Q_n$ consists of $\binom{d}{i}$ shapes where each shape is a $j_1 \times j_2 \times \cdots \times j_{i+1}$ parallelepiped where $j_1 + j_2 + \cdots + j_{i+1} = d+1$, and each shape occupies the same $d$ levels. The color classes in each shape correspond to the $Q_i$'s using the same set of $i$ stars.  Conversely, every instance of such a shape sequence where the shapes are arranged in a proper relative position corresponds to some embedding of $Q_d$ in $Q_n$.
\end{fact}

For example, if $Q_2$'s are colored, the embedding $[1111{*}011{*}10110{*}101{*}001]$ of $Q_4$ in $Q_{22}$ contains 24 $Q_2$'s, partitioned into $\binom{4}{2}=6$ sets according to which two stars they use.  These six sets give rise to six shapes, all parallelepipeds containing color classes in levels 12, 13, and 14 (see Figure~\ref{q4nq22}).

\begin{figure}
\begin{tabular}{l|l|l|l}
Sets of two stars  & $Q_2$'s using those stars&Color classes&Dimensions of shape\\ \hline

3rd and 4th & $[11110011010110{*}101{*}001]$ & $(9,2,1)$ & $3\times 1 \times 1$ \\
 & $[11110011110110{*}101{*}001]$ & $(10,2,1)$&\\
 & $[11111011010110{*}101{*}001]$ & $(10,2,1)$ &  \\
 & $[11111011110110{*}101{*}001]$ & $(11,2,1)$&\\ \hline
2nd and 4th & $[11110011{*}101100101{*}001]$ & $(6,5,1)$ & $2\times 2 \times 1$ \\
 & $[11110011{*}101101101{*}001]$ & $(6,6,1)$&\\ 
 & $[11111011{*}101100101{*}001]$ & $(7,5,1)$ &  \\
 & $[11111011{*}101101101{*}001]$ & $(7,6,1)$&\\ \hline
1st and 4th & $[1111{*}0110101100101{*}001]$ & $(4,7,1)$ & $1\times 3 \times 1$ \\
 & $[1111{*}0110101101101{*}001]$ & $(4,8,1)$&\\ 
 & $[1111{*}0111101100101{*}001]$ & $(4,8,1)$ &  \\
 & $[1111{*}0111101101101{*}001]$ & $(4,9,1)$&\\ \hline
 1st and 3rd & $[1111{*}011010110{*}1010001]$ & $(4,5,3)$ & $1\times 2 \times 2$ \\
 & $[1111{*}011010110{*}1011001]$ & $(4,5,4)$&\\ 
 & $[1111{*}011110110{*}1010001]$ & $(4,6,3)$ &  \\
 & $[1111{*}011110110{*}1011001]$ & $(4,6,4)$&\\ \hline
  2nd and 3rd & $[11110011{*}10110{*}1010001]$ & $(6,3,3)$ & $2\times 1 \times 2$ \\
 & $[11110011{*}10110{*}1011001]$ & $(6,3,4)$&\\ 
 & $[11111011{*}10110{*}1010001]$ & $(7,3,3)$ &  \\
 & $[11111011{*}10110{*}1011001]$ & $(7,3,4)$&\\ \hline
  1st and 2nd & $[1111{*}011{*}1011001010001]$ & $(4,2,6)$ & $1\times 1 \times 3$ \\
 & $[1111{*}011{*}1011001011001]$ & $(4,2,7)$&\\ 
 & $[1111{*}011{*}1011011010001]$ & $(4,2,7)$ &  \\
 & $[1111{*}011{*}1011011011001]$ & $(4,2,8)$&\\ 
\end{tabular}
\caption{The six shapes for the embedding $[1111{*}011{*}10110{*}101{*}001]$ of $Q_4$ in $Q_{22}$.}\label{q4nq22}
\end{figure}

\"Ozkahya and Stanton \cite{OS11} proved lower bounds on $p^i(Q_d)$ precisely analogous to those in Theorem~\ref{AKS}.   Suppose the shape with the most elements in the shape sequence for $Q_d$ contains $r$ color classes.  Then there is an $r$-coloring of the $Q_i$'s where the largest shape in the shape sequence contains all $r$ colors for every copy of $Q_d$.  The number of colors of this shape is the largest product of $i+1$ natural numbers that sum to $d+1$, and is obtained by minimizing the difference between these numbers.  The exact value depends on the remainder of $d+1 \pmod{i+1}$; if the remainder is zero, the number of colors is $\left(\frac{d+1}{i+1}\right)^{i+1}$. From this, one obtains lower bounds of $p^2(Q_3) \ge 1 \cdot 1 \cdot 2 = 2$, $p^2(Q_4) \ge 1 \cdot 2 \cdot 2 = 4$, $p^2(Q_5) \ge 2 \cdot 2 \cdot 2 = 8$, $p^2(Q_6) \ge 2 \cdot 2 \cdot 3 = 12$, and so forth.

\"Ozkahya and Stanton \cite{OS11} also proved the upper bound  analogous to Theorem~\ref{AKS},  which is $p^i(Q_d) \le \binom{d+1}{i+1}$.  This is the number of color classes covered by an embedding of $Q_d$ with $d$ stars and all other entries 0. For example, if $Q_2$'s are colored, the embedding of $Q_3$ represented by $[{*}{*}{*}00000\ldots]$ has $Q_2$'s only in color classes (0,0,0), (1,0,0), (0,1,0), and (0,0,1), and so $p^2(Q_3) \le 4 = \binom{3+1}{2+1}$.

In the case of edge colorings, the polychromatic number for $Q_d$ turned out to be the lower bound in Theorem~\ref{AKS}, but this is not the case for colorings of larger subcubes; the polychromatic number may be larger than the size of the largest shape.  We show this first in the case of $p^2(Q_3)$. All shapes in the shape sequence for $Q_3$ when $Q_2$'s are colored have two elements, and the bounds corresponding to Theorem~\ref{AKS} are $2 \le p^2(Q_3) \le 4$.

\begin{theorem}\label{p233}
 $p^2(Q_3)=3$.
\end{theorem}

\begin{proof}
Lower bound: Consider the simple 3-coloring $\chi$ that assigns colors to any $Q_2$ in color class $(x_1, x_2, x_3)$ the color $\chi(x_1, x_2, x_3)$, where
\[ \chi(x_1, x_2, x_3)= \begin{cases}
x_1+x_2+x_3 &\pmod{3} \hspace{.3in} \text{ if } x_2 \equiv  0 \pmod{2}\\
x_1+x_2+x_3+1 &\pmod{3} \hspace{.3in} \text{ if } x_2 \equiv  1 \pmod{2}.\\
\end{cases}\] 
Consider an embedding of $Q_3$ where there are $a_1$ 1's to the left of the first star, $a_2$ 1's between the first and second stars, $a_3$ 1's  between the second and third stars, and $a_4$ 1's to the right of the third star.  We show that it contains all three colors. Without loss of generality assume $a_1+a_2 +a_3 + a_4 \equiv 0 \pmod{3}$, and $a_1 = a_4 = 0$. For $1 \le i < j \le 3$, denote by $S_{ij}$ the shape using stars $i$ and $j$. The tables in Figure~\ref{p23tab} list the color classes in each of the three shapes, and the colors in each shape for the four possibilities for $a_2$ and $a_3$ $\pmod{2}$. For each possibility, each of the three colors is in at least one of the three shapes.

\begin{figure}
\begin{center}
\begin{tabular}{l|l}
Shape & Color  Classes \\ \hline
$S_{12}$ & $(0, a_2, a_3)$, $(0, a_2, a_3+1)$\\ \hline
$S_{13}$ & $(0, a_2+ a_3,0)$, $(0, a_2+ a_3+1,0)$\\ \hline
$S_{23}$ & $(a_2, a_3,0)$, $(a_2+1, a_3,0)$\\ 
\end{tabular}
\hspace{.5in}
  \begin{tabular}{l|l|l|l|l}
 $a_2$ & $a_3$ & $S_{12}$ & $S_{13}$ & $S_{23}$ \\ \hline
0&0& 0,1&0,2&0,1\\ \hline
0&1& 0,1 & 1 & 1,2 \\ \hline
1&0& 1,2 & 1 & 0,1 \\ \hline
1&1& 1,2 & 0,2 & 1,2\\ 
\end{tabular}
\end{center}
\caption{The colors of the $Q_2$'s contained in each shape for the 4 possibilities for $a_2$ and $a_3$ $\pmod{2}$ in an embedding of $Q_3$, using the coloring $\chi$ from Theorem~\ref{p233}.}\label{p23tab}
 \end{figure}

Upper bound: A computer search found that no $Q_3$-polychromatic 4-coloring, simple or otherwise, is possible on $Q_5$. 

For a human-checkable proof, we know that if the $Q_2$'s in $Q_n$ can be 4-colored so that every $Q_3$ is polychromatic, then it can be done  with a simple coloring. We try to construct a simple coloring $\chi$ on $Q_5$, eventually showing that it is impossible.

For the sake of notational compactness, in this proof color classes are represented by strings rather than tuples.  For example we will represent that color class $(1,0,1)$ by the string 101. The embedding $[{*}{*}{*}00]$ of $Q_3$ contains only color classes 000, 100, 010, and 001, so these color classes must all be distinct, and without loss of generality we assign them the colors 1, 2, 3, and 4, respectively. 

To assign colors to the color classes 110, 101, 011, 200, 020, and 002, we examine the embeddings $[1{*}{*}{*}0]$, $[{*}1{*}{*}0]$, $[{*}{*}1{*}0]$, and $[{*}{*}{*}10]$.  A straightforward but tedious examination shows there are only five simple 4-colorings of the $Q_4$ $[{*}{*}{*}{*}0]$, falling into two patterns:  One is $\chi_1$, below.  The other four are $\chi_2$, where $\chi_2(020)$ can be chosen to be any color.

\begin{center}
\begin{tabular}{l|l|l|l|l|l|l|l|l|l|l}
Color class & 000 & 100 & 010 & 001 & 101 & 011 & 002 & 110 & 020 & 200  \\ \hline
$\chi_1$ & 1 & 2 & 3 & 4  &1 & 3 & 2 & 3 & 1 & 4 \\ \hline
$\chi_2$ & 1 & 2 & 3 & 4  &3 & 1 & 2 & 1 & {*} & 4 \\
\end{tabular}
\end{center}
  
We now attempt to extend these colorings to the $Q_4$'s $[1{*}{*}{*}{*}]$ and $[{*}{*}{*}{*}1]$. In the case of $[1{*}{*}{*}{*}]$, we already know the colors of 100, 200, 110, and 101, so the colors for the color classes for $[1{*}{*}{*}{*}]$ must fit one of the following four patterns. For $i,j \in \{1,2\}$, $\chi_{ij}$ is the coloring generated by extending the $\chi_i$ coloring with the $\chi_j$ pattern.

\begin{center}
\begin{tabular}{l|l|l|l|l|l|l|l|l|l|l}
Color class & 100 & 200 & 110 & 101 & 201 & 111 & 102 & 210 & 120 & 300  \\ \hline
$\chi_{11}$ & 2 & 4 & 3 & 1  &2 & 3 & 4 & 3 & 2 & 1 \\ \hline
$\chi_{12}$ & 2 & 4 & 3 & 1  &3 & 2 & 4 & 2 & {*} & 1 \\ \hline
$\chi_{21}$ & 2 & 4 & 1 & 3  &2 & 1 & 4 & 1 & 2 & 3 \\ \hline
$\chi_{22}$ & 2 & 4 & 1 & 3  &1 & 2 & 4 & 2 & {*} & 3 \\ 
\end{tabular}
\end{center}
In the $Q_4$ $[{*}{*}{*}{*}1]$, color classes 001, 101, 011, 002, 102, 111, and 201 are already assigned. It is impossible to extend $\chi_{12}$ and $\chi_{22}$ consistent with the patterns $\chi_1$ or $\chi_2$, while $\chi_{11}$ and $\chi_{21}$ extend uniquely as follows.

\begin{center}
\begin{tabular}{l|l|l|l|l|l|l|l|l|l|l}
Color class & 001 & 101 & 011 & 002 & 102 & 012 & 003 & 111 & 021 & 201  \\ \hline
$\chi_{11}$ & 4 & 1 & 3 & 2  &4 & 3 & 1 & 3 & 4 & 2 \\ \hline
$\chi_{21}$ & 4 & 3 & 1 & 2  &4 & 1 & 3 & 1 & 4 & 2 \\
\end{tabular}
\end{center}
Now every color class except for 030 is assigned a color.  However in either of the colorings, the embedding $[{*}1{*}1{*}]$ (containing color classes 110, 210, 020, 030, 011, and 012) has at most two colors before 030 is assigned.  Thus it cannot contain four colors regardless of the choice for the color of 030.
\end{proof}

\begin{theorem}\label{p24}
$p^2(Q_4) \ge 5$.
\end{theorem}

\begin{proof}
Consider the simple 5-coloring $\chi$ that assigns to any $Q_2$ in color class $(x_1, x_2, x_3)$ the color $\chi(x_1, x_2, x_3)$, where

\[\chi(x_1,x_2,x_3) = \begin{cases}

x_1 + x_2 + x_3 &\pmod{5} \hspace{.3in} \text{ if } x_2 \equiv 0 \pmod{3}\\
x_1 + x_2 + x_3 + 1 &\pmod{5} \hspace{.3in}  \text{ if } x_2 \equiv 1 \pmod{3}\\
x_1 + x_2 + x_3 + 2 &\pmod{5} \hspace{.3in}  \text{ if } x_2 \equiv 2 \pmod{3}.\\

\end{cases} \]

Consider an embedding of $Q_4$ where there are $a_1$ 1's to the left of the first star, $a_2$ 1's between the first and second stars, $a_3$ 1's  between the second and third stars, $a_4$ 1's  between the third and fourth stars, and $a_5$ 1's to the right of the fourth star.  We show that it contains all five colors. Without loss of generality assume $a_1+a_2 +a_3 + a_4 + a_5 \equiv 0 \pmod{5}$ and $a_1 = a_5 = 0$.  For $1 \le i < j \le 4$, denote by $S_{ij}$ the shape using stars $i$ and $j$. The following table lists the color classes in each of the six shapes.

\begin{tabular}{l|l}
Shape & Color Classes \\ \hline
$S_{12}$ &  $(0, a_2, a_3+a_4)$,  $(0, a_2, a_3+a_4+1)$, $(0, a_2, a_3+a_4+2)$ \\ \hline
$S_{13}$ &  $(0, a_2+ a_3,a_4)$, $(0, a_2+ a_3+1,a_4)$, $(0, a_2+ a_3,a_4+1)$, $(0, a_2+ a_3+1,a_4+1)$ \\ \hline
$S_{14}$ &  $(0, a_2+ a_3+a_4,0)$,  $(0, a_2+ a_3+a_4+1,0)$, $(0, a_2+ a_3+a_4+2,0)$ \\ \hline
$S_{24}$ &  $(a_2,a_3+a_4,0)$, $(a_2+1, a_3+a_4,0)$, $(a_2,a_3+a_4+1,0)$, $(a_2+1, a_3+a_4+1,0)$ \\ \hline
$S_{34}$ &  $(a_2+a_3,a_4,0)$,  $(a_2+a_3+1,a_4,0)$, $(a_2+ a_3+2,a_4,0)$\\ \hline
$S_{23}$ &  $(a_2,a_3,a_4)$,   $(a_2+1, a_3,a_4)$, $(a_2,a_3,a_4+1)$, $(a_2+1, a_3,a_4+1)$ 
\end{tabular}
 
Figure~\ref{p24tab} lists the colors contained in each shape for the 27 possibilities for $a_2$, $a_3$, and $a_4$ $\pmod{3}$. For each possibility, each of the five colors is in at least one of the six shapes.
\end{proof}

\begin{figure}
\begin{center}
\begin{tabular}{l|l|l|l|l|l|l|l|l}
$a_2$ & $a_3$ & $a_4$ & $S_{12}$ & $S_{13}$ & $S_{14}$ & $S_{24}$ & $S_{34}$ & $S_{23}$ \\ \hline
0&0&0 &0,1,2&0,1,2,3&0,2,4&0,1,2,3&0,1,2&0,1,2\\ \hline
1&1&1 &1,2,3&2,1,3&0,2,4&2,1,3&1,2,3&1,2,3\\ \hline
2&2&2 &2,3,4&1,2,3,4&0,2,4&1,2,3,4&2,3,4&2,3,4\\ \hline
1&0&0 &1,2,3&1,2,3,4&1,3,2&0,1,2,3&0,1,2&0,1,2\\ \hline
0&1&0 &0,1,2&1,2,3,4&1,3,2&1,2,3,4&0,1,2&1,2,3\\ \hline
0&0&1 &0,1,2&0,1,2,3&1,3,2&1,2,3,4&1,2,3&0,1,2\\ \hline
1&1&0 &1,2,3&2,1,3&2,1,3&1,2,3,4&0,1,2&1,2,3\\ \hline
1&0&1 &1,2,3&1,2,3,4&2,1,3&1,2,3,4&1,2,3&0,1,2\\ \hline
0&1&1 &0,1,2&1,2,3,4&2,1,3&2,1,3&1,2,3&1,2,3\\ \hline
2&0&0 &2,3,4&2,1,3&2,1,3&0,1,2,3&0,1,2&0,1,2\\ \hline
0&2&0 &0,1,2&2,1,3&2,1,3&2,1,3&0,1,2&2,3,4\\ \hline
0&0&2 &0,1,2&0,1,2,3&2,1,3&2,1,3&2,3,4&0,1,2\\ \hline
2&1&0 &2,3,4&2,1,3&0,2,4&1,2,3,4&0,1,2&1,2,3\\ \hline
1&2&0 &1,2,3&0,1,2,3&0,2,4&1,2,3,4&0,1,2&2,3,4\\ \hline
2&0&1 &2,3,4&2,1,3&0,2,4&1,2,3,4&1,2,3&0,1,2\\ \hline
1&0&2 &1,2,3&1,2,3,4&0,2,4&2,1,3&2,3,4&0,1,2\\ \hline
0&2&1 &0,1,2&2,1,3&0,2,4&0,1,2,3&1,2,3&2,3,4\\ \hline
0&1&2 &0,1,2&1,2,3,4&0,2,4&0,1,2,3&2,3,4&1,2,3\\ \hline
2&1&1 &2,3,4&0,1,2,3&1,3,2&2,1,3&1,2,3&1,2,3\\ \hline
1&2&1 &1,2,3&0,1,2,3&1,3,2&0,1,2,3&1,2,3&2,3,4\\ \hline
1&1&2 &1,2,3&2,1,3&2,1,3&0,1,2,3&2,3,4&1,2,3\\ \hline
2&2&0 &2,3,4&1,2,3,4&1,3,2&2,1,3&0,1,2&2,3,4\\ \hline
2&0&2 &2,3,4&2,1,3&1,3,2&2,1,3&2,3,4&0,1,2\\ \hline
0&2&2 &0,1,2&2,1,3&1,3,2&2,1,3&2,3,4&2,3,4\\ \hline
1&2&2 &1,2,3&0,1,2,3&1,3,2&1,2,3,4&2,3,4&2,3,4\\ \hline
2&1&2 &2,3,4&0,1,2,3&1,3,2&0,1,2,3&2,3,4&1,2,3\\ \hline
2&2&1 &2,3,4&1,2,3,4&1,3,2&0,1,2,3&1,2,3&2,3,4\\ 

\end{tabular}
\end{center}
\caption{The colors of the $Q_2$'s contained in each shape for the 27 possibilities for $a_2$, $a_3$, and $a_4$ $\pmod{3}$ in an embedding of $Q_4$, using the coloring $\chi$ from Theorem~\ref{p24}.}\label{p24tab}
\end{figure}

\begin{theorem}\label{pdproj}
For all $d \ge i \ge 1$, $j \ge 1$, $p^{i+j}(Q_{d+j})\ge p^{i}(Q_d)$.
\end{theorem}
\begin{proof}
Suppose $\chi$ is a $Q_d$-polychromatic $k$-coloring of the $Q_{i}$'s in $Q_n$, where $n \ge d+j$. Consider the $k$-coloring of the $Q_{i+j}$'s in $Q_n$ given by
\[\chi'(x_1, x_2, \ldots, x_{i+1}, \ldots, x_{i+j+1}) = \chi(x_1, x_2, \ldots, x_{i+1}).\]
We show $\chi'$ is $Q_{d+j}$-polychromatic.

Let $G_{d+j}$ be an embedding of $Q_{d+j}$ in $Q_n$ represented by an $n$-bit vector with $d+j$ stars. Let $G_d$ be the embedding of $Q_d$ in $Q_n$ represented by the same vector with the $(d+1)$st star and every coordinate to the right replaced by zeros. For example, if $d=4$, $j=2$, and $G_{d+j} = [0{*}111{*}1011{*}{*}110{*}0{*}01]$, then $G_{d} = [0{*}111{*}1011{*}{*}11000000]$. 
The $i+1$ coordinates of the color classes for $G_d$ are identical to the first $i+1$ coordinates of the color classes for $G_{d+j}$ in the shapes that use the last $j$ stars. Since $\chi$ is $Q_d$-polychromatic, $G_d$ contains $Q_{i}$'s of each of the $k$ colors, and $G_{d+j}$ must also contain a $Q_{i+j}$ of each color with the coloring $\chi'$.
\end{proof}



\begin{corollary}
For all $d \ge 2$, $p^d(Q_{d+1}) \ge p^2(Q_3) = 3$.
\end{corollary}


It is not clear whether the colorings in Theorems~\ref{p233} and \ref{p24} have natural generalizations, or whether they are sporadic small cases.  Thus almost any progress in determining polychromatic numbers when larger subcubes are colored would be interesting. 

\begin{problem}
Improve any of the following bounds:

\begin{itemize}
\item For $d\ge 3$, $d+2 \ge p^d(Q_{d+1})\ge 3$. 
\item $10 \ge p^2(Q_4) \ge 5$.
\item For $d\ge 5$, $\binom{d+1}{3}\ge p^2(Q_d) \ge p(Q_{d-1})$.
\end{itemize}
\end{problem}

Added in proof:  Recently Chen~\cite{Chen16} has improved the lower bound on $p^2(Q_d)$ for all $d \ge 4$.

\begin{theorem}[Chen~\cite{Chen16}]\label{Chenp2}
For all $d \ge 4$,
\[ p^2 (Q_d) \ge \begin{cases} 
  (k^2+1)(k+1) & d=3k \\ 
  (k^2+k+1)(k+1) & d=3k +1\\ 
    (k^2+k+1)(k+2) & d=3k +2. 
  \end{cases} \]

\end{theorem}

\section{Acknowledgements}

This research was conducted at a workshop made possible by the Alliance for Building Faculty Diversity in the Mathematical Sciences (DMS 0946431), held at the Institute for Computational and Experimental Research in Mathematics.  The fourth author was supported by the Westminster College McCandless Research Award and would like to thank Reshma Ramadurai for helpful conversations.  All authors would like to thank Maria Axenovich for collaboration.

\end{document}